\documentclass[12pt]{amsart}
\usepackage{amssymb}
\usepackage{verbatim}
\usepackage[usenames]{color}
\usepackage{hyperref}
\usepackage{url}

\newtheorem{theorem}{Theorem}[section]
\newtheorem{proposition}[theorem]{Proposition} 
 
\newtheorem{lemma}[theorem]{Lemma}

\newtheorem{fact}[theorem]{Fact}

\theoremstyle{definition}
\newtheorem{definition}[theorem]{Definition}

\newtheorem{notation}[theorem]{Notation}

\theoremstyle{remark}

\newcommand{\defemp}{\textit}

\newcommand{\forces}{\Vdash}
\newcommand{\zfc}{\textrm{ZFC}}
\newcommand{\zf}{\textrm{ZF}}

\newcommand{\mc}{\mathcal}
\newcommand{\mbb}{\mathbb}

\newcommand{\p}{\mathcal{P}}

\newcommand{\baire}{{{^\omega}\omega}}
\newcommand{\bairenodes}{{^{<\omega}\omega}}
\newcommand{\cantornodes}{{^{<\omega}2}}

\newcommand{\im}{\textnormal{Im}}
\newcommand{\dom}{\textnormal{Dom}}

\title[Sacks Forcing and the Shrink Wrapping Property]
 {Sacks Forcing and \\ the Shrink Wrapping Property}
\author{Osvaldo Guzm\'{a}n Gonz\'{a}lez}
\address{Mathematics Department\\
 University of Toronto\\
 40 St. George St., Toronto Canada}
\email{oguzman@math.utoronto.ca}

\author{Dan Hathaway}
\address{Mathematics Department\\
University of Vermont\\
16 Colchester Ave., Burlington VT U.S.A.}
\email{danhath@umich.edu}

\begin{document}

\begin{abstract}
We consider a property stronger than
 the Sacks property,
 called the
 \textit{shrink wrapping} property,
 which holds between
 the ground model and each Sacks forcing extension.
Unlike the Sacks property, the shrink wrapping property
 does not hold between the ground model
 and a Silver forcing extension.
We also show an application
 of the shrink wrapping property.
\end{abstract}

\maketitle

\section{The Shrink Wrapping Property}

Within this section,
 we will define the \textit{shrink wrapping property},
 which is a strengthening of the Sacks property,
 which holds between each Sacks forcing extension
 and the ground model,
 but not between any Silver forcing extension
 and the ground model.


\begin{definition}
Given a function $f : \omega \to (\omega - \{0\})$,
 we say that a tree
 $T \subseteq \bairenodes$
 \defemp{obeys} $f$ iff
 for each $l \in \omega$,
 the set $$\{ n \in \omega :
 t^\frown n \in T \mbox{ for some }
 t \mbox{ on level } l \mbox{ of } T \}$$
 has size $\le f(l)$.
\end{definition}

\begin{definition}
Let $M$ be a transitive model of $\zf$.
The \defemp{Sacks property} holds between $V$ and $M$
 iff given any function $f : \omega \to (\omega - \{0\})$
 in $M$ satisfying $\lim_{l \to \infty} f(l) = \infty$
 and given any $x \in \baire$ (in $V$),
 there is some tree $T \in M$ which obeys $f$
 such that $x \in [T]$.
\end{definition}

The Sacks property as we have just defined it is equivalent
 to the version where we only consider a single such function
 $f$ in $M$, instead of all such functions.
Suppose that $V$ is a Sacks forcing extension of
 a model $M$.
Then the \textit{Sacks property} holds between $V$ and $M$.

Now, if $V$ is a Sacks forcing extension of $M$
 and $\langle x_n: n < \omega \rangle$
 is a sequence
 of reals (in $V$) 
 then we cannot expect
 there to be a single function $f$ in $M$
 and a sequence of trees
 $\langle T_n  : n \in \omega \rangle$
 in $M$ such that for each $n$,
 $T_n$ obeys $f$ and $x_n \in [T_n]$.
To see why this is case,
 suppose that all $T_n$ obey the same function $f$.
Let $k = f(0)$.
Then $\mathcal{S} := \langle x_n(0) : n < \omega \rangle$
 is such that there is a function
 $g : \omega \to [\omega]^k$ in $M$ satisfying
 $(\forall n \in \omega)\, x_n(0) \in g(n)$.
However, $\mathcal{S}$ could be such that this
 is impossible.

To remedy this problem, fix the following
 sequence of functions:
\begin{notation}
By $\langle f_i : i < \omega \rangle$,
 we will denote a sequence of functions such that $f_i$ is a function from $\omega \to (\omega - \{0\})$
 and $(i,l) \mapsto f_i(l)$ is computable and finite-to-one.
\end{notation}
Since $(i,l) \mapsto f_i(l)$ is computable,
 it is contained in every model of $\zf$.
Note that $(i,l) \mapsto f_i(l)$ being finite-to-one
 implies that for each fixed $i$,
 $\lim_{l \to \infty} f_i(l) = \infty$.
Also, if $\langle (i_j, l_j) : j < \omega \rangle$
 is any enumeration of $\omega \times \omega$
 without repetation, then
$$\lim_{j \to \infty} f_{i_j}(l_j) = \infty.$$
This allows us to use the Sacks property.
Here is what we mean:

Suppose the Sacks property holds
 between $V$ and $M$
 and $\langle x_i : i \in \omega \rangle$
 is any sequence of reals.
Then there is a sequence
 of trees $\langle T_i : i < \omega \rangle \in M$
 such that $(\forall n < \omega)\,$
 $T_i$ obeys $f_i$ and $x_i \in [T_i]$.
This is because by the Sacks property,
 using the enumeration $\langle (i_j, l_j) : j < \omega \rangle$ above,
 there is a function $h : \omega \to [\omega]^{<\omega}$
 such that for each $j < \omega$,
 $|h(j)| = f_{i_j}(l_j)$ and
 $x_{i_j}(l_j) \in h(j)$.
The function $h$ can then be used to create
 the sequence of trees $T_i$ for $i < \omega$.

The following is a stronger property that we might want
 to hold between $V$ and $M$:
 for every sequence of reals 
 $\mc{X} = \langle x_n  : n < \omega \rangle$
 there exists a sequence of trees
 $\langle T_n  :
 n < \omega \rangle \in M$ such that
\begin{itemize}
\item[1)] $(\forall n \in \omega)$
 $T_n$ obeys $f_n$ and $x_n \in [T_n]$;
\item[2)] $(\forall n_1, n_2 \in \omega)\,$ one of the following holds:
\begin{itemize}
\item[a)] $x_{n_1} = x_{n_2}$;
\item[b)] $[T_{n_1}] \cap [T_{n_2}] = \emptyset$.
\end{itemize}
\end{itemize}
Unfortunately, if the sequence
 $\mc{X}$ satisfies
 $$\langle (n_1, n_2) : x_{n_1} = x_{n_2} \rangle \not\in M,$$
 then there can be no such sequence of trees in $M$.
Thus, we need a weaker notion:
 a \textit{shrink wrapper}.

\begin{notation}
By $\eta : \omega \to [\omega]^2$ we will denote a computable bijection 
 so that for each
 $\tilde{n} \in \omega$,
 we may talk about the
 $\tilde{n}$-th pair
 $\eta(\tilde{n}) \in [\omega]^2$.
\end{notation}
The idea of a shrink wrapper
 is that for each
 $\{ n_1, n_2 \} = \eta(\tilde{n}) \in [\omega]^2$,
 the functions
 $F_{\tilde{n},n_1}$ and
 $F_{\tilde{n},n_2}$,
 together with the finite sets
 $I(n_1)$ and $I(n_2)$,
 will separate $x_{n_1}$
 and $x_{n_2}$ as much as possible.
For $n \in \eta(\tilde{n})$,
 the function
 $F_{\tilde{n},n} : {^{\tilde{n}}2} \to \p(\bairenodes)$
 is shrink-wrapping $2^{\tilde{n}}$
 possibilities for the value of $x_n$.
We need to make sure that
 what contains one possibility for $x_{n_1}$
 is sufficiently disjoint from
 what contains another possibility for $x_{n_2}$,
 even if it is not possible that
 simultaneously both $x_{n_1}$
 and $x_{n_2}$
 are in the respective containers.

Fix $\tilde{n} \in \omega$
 and consider the $\tilde{n}$-th pair
 $\{ n_1, n_2 \}$.
If $x_{n_1} = x_{n_2}$, they certainly
 cannot be separated and this
 is a special case.
Also,
 there are finitely many ``isolated'' points
 which might prevent the separation of
 $x_{n_1}$ from $x_{n_2}$.
In fact, we can get a finite set $I(k)$ of isolated
 points associated to each $x_k$ as opposed
 to each pair $\{ x_{n_1}, x_{n_2} \}$.

When we construct a shrink wrapper for
 a sequence of reals in a Sacks forcing extension,
 we can easily get the trees that occur in
 the shrink wrapper to obey functions in the ground
 model.
To facilitate this,
 we do the following:
\begin{notation}
By $\Phi : {^{<\omega}2} \times \omega \to \omega$ 
we will denote a fixed a finite-to-one function.
\end{notation}
In the definition of a shrink wrapper, we will have
 each $F_{\tilde{n},n}(s)$ be a tree which obeys
 $f_{\Phi(s,n)}$.
Thus, the definition of a shrink wrapper depends
 on the finite-to-one function $\Phi$ and the
 finite-to-one function
 $( i, l ) \mapsto f_i(l)$.
However, the reader can check that the
 choice of these two functions is not important,
 as long as they are both in the ground model
 and are both finite-to-one.
Note that
 $(s,n,l) \mapsto f_{\Phi(s,n)}(l)$ is finite-to-one,
 which allows us to use the Sacks property.

\begin{definition}
\label{srdef}
A \defemp{shrink wrapper} $\mc{W}$ for
a sequence of reals $\mc{X} = \langle x_n : n \in \omega \rangle$
 is a pair $\langle \mc{F}, I \rangle$ such that
 $I : \omega \to [\baire]^{<\omega}$ and
 $\mc{F}$ is a collection of functions
 $F_{\tilde{n},n}$ for $\tilde{n} \in \omega$
 and $n \in \eta(\tilde{n})$ which satisfy
 the following
 conditions.
\begin{itemize}
\item[1)]
Given $\tilde{n}$ and $n \in \eta(\tilde{n})$,
 $F_{\tilde{n},n} : {^{\tilde{n}}2} \to \p(\bairenodes)$ and
 for each $s \in {^{\tilde{n}}2}$,
 $F_{\tilde{n},n}(s) \subseteq \bairenodes$
is a leafless tree that obeys $f_{\Phi(s,n)}$.
\item[2)]
Given $\tilde{n}$ and $n \in \eta(\tilde{n})$,
 $(\exists s \in {^{\tilde{n}}2})\,
 x_n \in [F_{\tilde{n},n}(s)]$.
\end{itemize}
\begin{itemize}
\item[3)]
Given
 $\{n_1, n_2\} = \eta(\tilde{n})$, 
 $(\forall s_1, s_2 \in {^{\tilde{n}}2})$
 one of the following relationships holds
 between the sets
 $C_1 := [F_{\tilde{n},n_1}(s_1)]$ and
 $C_2 := [F_{\tilde{n},n_2}(s_2)]$:
\begin{itemize}
\item[3a)] $C_1 = C_2$ and if either
 $x_{n_1} \in C_1$ or
 $x_{n_2} \in C_2$, then
 $x_{n_1} = x_{n_2}$;
\item[3b)] $(\exists x \in I(n_1) \cap I(n_2))\,
 C_1 = C_2 = \{x\}$;
\item[3c)] $C_1 \cap C_2 = \emptyset$,
 and therefore $(\exists l \in \omega)
 (\forall (y_1, y_2) \in C_1 \times C_2)\,
 y_1$ and $y_2$ differ before level $l$.
\end{itemize}
\end{itemize}
\end{definition}

The therefore part of 3c) is because if for each $l$
 there was a node on level $l$ of
 the tree $T :=
 F_{\tilde{n},n_1}(s_1) \cap
 F_{\tilde{n},n_2}(s_2)$,
 then because $T$ has finite branching,
 by K\"onig's lemma it would have an infinite branch.
When we construct a shrink wrapper,
 we can usually ensure that it
 satisfies the following additional property:
\begin{itemize}
\item[4)]
Given $\tilde{n}$ and $n \in \eta(\tilde{n})$,
 $(\forall s_1, s_2 \in {^{\tilde{n}}2})$
 one of the following relationships holds between
 the sets $C_1 := [F_{\tilde{n},n}(s_1)]$
 and $C_2 := [F_{\tilde{n},n}(s_2)]$:
\begin{itemize}
\item[4a)] $(\exists x \in I(n))\, C_1 = C_2 = \{x\}$;
\item[4b)] $C_1 \cap C_2 = \emptyset$,
 and therefore $(\exists l \in \omega)
 (\forall (y_1, y_2) \in C_1 \times C_2)\,
 y_1$ and $y_2$ differ before level $l$.
\end{itemize}
\end{itemize}
Note this is a requirement
 on the single function $F_{\tilde{n},n}$
 where $n \in \eta(\tilde{n})$,
 and not a requirement on the pair of functions
 $( F_{\tilde{n},n_1}, F_{\tilde{n},n_2} )$
 where $\{n_1, n_2\} = \eta(\tilde{n})$.

\begin{definition}
Given a model $M$ of $\zfc$,
 we say that the \defemp{shrink wrapping property}
 holds between $M$ and $V$ iff
 every sequence of reals $\mc{X} = \langle x_n :
 n \in \omega \rangle$ has
 a shrink wrapper $\mc{W}$ in $M$.
A forcing $\mbb{P}$ has the shrink wrapping property
 iff the shrink wrapping property holds between
 the ground model and each forcing extension.
\end{definition}

In Theorem~\ref{mainsacksforcingargument}
 we will show that Sacks forcing has the
 shrink wrapping property.
If a forcing has the shrink wrapping property,
 then it automatically has the Sacks property.
That is, consider any real $x$ in the forcing extension.
Consider the sequence $\mc{X} = \langle x_n : n \in \omega \rangle$
 where $x_n = x$ for all $n$.
Let $\mc{W}$ be a shrink wrapper for $\mc{X}$
 and let $m \in \eta(0)$.
Now $x = x_m \in [F_{0,m}(\emptyset)]$,
 because $\emptyset$ is the only $s$ in ${^0 2}$.
Furthermore, $F_{0,m}(\emptyset) \subseteq \bairenodes$
 is a tree that
 obeys $f_{\Phi(\emptyset,0)}$
 (and that function does not depend on $\mc{X}$).

\section{Application to Pointwise Eventual Domination}

Before we show that there is always a shrink wrapper
 in the ground model after doing Sacks forcing,
 let us discuss an application of shrink wrappers themselves.
Given two functions $f, g : \baire \to \baire$,
 let us write $f \le^* g$ and say that $g$
 \textit{pointwise eventually dominates} $f$ iff
 $$(\forall x \in \baire)(\forall ^\infty n)\, f(x)(n) \le g(x)(n).$$
One may ask what is the cofinality of the set of Borel functions
 from $\baire$ to $\baire$ ordered by $\le^*$.
The answer is $2^\omega$,
 which follows from the result in \cite{Em} that
 given any $A \subseteq \omega$, there is a Baire class one
 (and therefore Borel) function $f_A : \baire \to \baire$
 such that given any Borel $g : \baire \to \baire$
 satisfying $f_A \le^* g$,
 we have that $A$ is $\Delta^1_1$ in any code for $g$.
One may ask what functions $f_A$ have such a property.

Being precise, say that a function $f : \baire \to \baire$
 \textit{sufficiently encodes} $A \subseteq \omega$ iff
 whenever $g : \baire \to \baire$ is Borel and 
 satisfies $f \le^* g$, then
 $A \in L[c]$ where $c$ is any code for $g$.
What must a function do to sufficiently encode $A$?
Given a sequence of reals 
 $\mc{X} = \langle x_n : n < \omega \rangle$,
 let us write $f_\mc{X} : \baire \to \baire$ for the function
 $$f_{\mc{X}}(x)(n) := \begin{cases}
 \min \{ l : x(l) \not= x_n(l) \} & \mbox{if } x \not= x_n, \\
 0 & \mbox{otherwise. }
 \end{cases}$$
Given $A \subseteq \omega$,
 is there always some $\mc{X}$ such that $f_\mc{X}$
 sufficiently encodes $A$?
It might seem like the answer is yes,
 because if a Borel function $g : \baire \to \omega$
 everywhere dominates one of the sections
 $x \mapsto f_\mc{X}(x)(n)$,
 then $x_n$ is $\Delta^1_1$ in any code for $g$
 \cite{Em}.

However, using a shrink wrapper, we can show that
 consistently there is not always a function of the form
 $f_\mc{X}$ that sufficiently encodes $A$.
Specifically,
 suppose $V$ is a Sacks forcing extension of
 an inner model $M$, $A \not\in M$,
 and $\mc{X}$ is a sequence of reals.
In the next section,
 we will show that there is
 a shrink wrapper $\mc{W} \in M$ for $\mc{X}$.
In this section we will show how to build from $\mc{W}$
 a Borel function $g : \baire \to \baire$,
 with a code $c \in M$, satisfying
 $f_\mc{X} \le^* g$.
Since $c \in M$,
 also $L[c] \subseteq M$,
 which implies $A \not\in L[c]$.
Hence, $f_\mc{X}$ does not sufficiently encode $A$.

To facilitate the discussion, let us make the following
 definitions:
\begin{definition}
Give $x \in \baire$,
 $[[x]] \subseteq \bairenodes$
 is the set of all finite
 initial segments of $x$.
\end{definition}
\begin{definition}
Given a tree $T \subseteq \bairenodes$,
 $\mbox{Exit}(T) : \baire \to \omega$ is the function
 $$\mbox{Exit}(T)(x) :=
 \begin{cases}
 \min \{ l : x \restriction l \not\in T \}
 & \mbox{if } x \not\in [T], \\
 0 & \mbox{if } x \in [T].
 \end{cases}$$
\end{definition}
For the remainder of this section we will show that if
 $M$ is a transitive model of $\zf$
 and a sequence $\mc{X}$ of reals has a
 shrink wrapper in $M$,
 then there is a Borel function $g$ with a code in $M$
 such that $f_\mc{X} \le^* g$.
We will illustrate the main ideas
 by considering a situation where $M$ contains
 something stronger than a shrink wrapper
 for $\mc{X}$.

\begin{proposition}
\label{treeavoid}
Let $M$ be a transitive model of $\zf$.
Let $$\mc{X} =
 \langle x_n  : n \in \omega \rangle$$ be a sequence of reals.
Suppose
 $$\mc{T} = \langle T_n : n \in \omega \rangle \in M$$ is a sequence
 of subtrees of $\bairenodes$ satisfying
 the following:
\begin{itemize}
\item[1)] $(\forall n \in \omega)\, x_n \in [T_n]$.
\item[2)] $(\forall n_1, n_2 \in \omega)$
 one of the following holds:
\begin{itemize}
\item[a)] $x_{n_1} = x_{n_2}$;
\item[b)] $[T_{n_1}] \cap [T_{n_2}] = \emptyset$.
\end{itemize}
\end{itemize}
Then there is a Borel function
 $g: \baire \to \baire$
 that has a Borel code in $M$ satisfying
 $$(\forall x \in \baire)\, f_\mathcal{X}(x) \le^* g(x).$$
\end{proposition}
\begin{proof}
Let $g: \baire \to \baire$ be defined by
 $$g(x)(n) := \max \{
  \mbox{Exit}(T_n)(x), n \}.$$
Certainly
 $g$ is Borel,
 with a code in $M$
 (because $\mc{T} \in M$).
The ``$\mbox{Exit}(T_n)(x)$'' part of the definition
 is doing most of the work.
Specifically, for any $n \in \omega$ and $x \not\in [T_n]$,
 $$f_\mc{X}(x)(n) =
 \mbox{Exit}([[x_n]])(x) \le
 \mbox{Exit}(T_n)(x).$$
This is because since $x_n$ is a path through
 the tree $T_n$,
 $x \not\in [T_n]$ implies
 the level where $x$ exits $T_n$
 is not before the level where $x$ differs from $x_n$.
Thus, we have
 $$(\forall n \in \omega)\,
 x \not\in [T_n] \Rightarrow
 f_\mc{X}(x)(n) \le
 g(x)(n).$$

Suppose, towards a contradiction, that there is some
 $x \in \baire$ satisfying $f_\mc{X}(x) \not\le^* g(x).$
Fix such an $x$.
Let $A$ be the infinite set
 $$A := \{ n \in \omega: f_\mc{X}(x)(n) > g(x)(n) \}.$$
It must be that
 $x \in [T_n]$ for each $n \in A$.
By hypothesis, this implies
 $x_{n_1} = x_{n_2}$ for all $n_1,n_2 \in A$.
Thus,
 $f_\mc{X}(x)(n)$
 is the same constant for all $n \in A$.
This is a contradiction, because
 $g(x)(n) \ge n$ for all $n$.
\end{proof}

Here is the stronger result where
 we only assume that $M$ has a shrink
 wrapper for $\mc{X}$:

\begin{theorem}
\label{combosepdevice}
Let $M$ be a transitive model of $\zf$.
Let $$\mc{X} =
 \langle x_n : n \in \omega \rangle$$ be a sequence of reals.
Suppose $\mc{W} = \langle \mc{F}, I \rangle \in M$
 is a shrink wrapper for $\mc{X}$.
Then there is a Borel function $g : \baire \to \baire$
 that has a Borel code in $M$ satisfying
 $$(\forall x \in \baire)\,
 f_\mc{X}(x) \le^* g(x).$$
\end{theorem}
\begin{proof}
For each $n \in \omega$,
 let $T_n \subseteq \bairenodes$ be the tree
 $$T_n := \bigcap
 \{ \bigcup \im(F_{\tilde{n},n}) :
 \tilde{n} \in \omega \wedge
 n \in \eta(\tilde{n}) \}.$$
That is, for each $t \in \bairenodes$,
 $t \in T_n$ iff 
 $$(\forall \tilde{n} \in \omega)[
 n \in \eta(\tilde{n}) \Rightarrow
 t \in \bigcup_{s \in {^{\tilde{n}}2}} F_{\tilde{n},n}(s)].$$
By part 2) of the definition of
 a shrink wrapper,
 $$(\forall n \in \omega)\,
 x_n \in [T_n].$$
Let $e(n_2)$ be the least
 level $l$ such that if
 $n_1 < n_2$,
 $\tilde{n}$ satisfies $\eta(\tilde{n}) = \{n_1, n_2\}$,
 and $s_1,s_2 \in {^{\tilde{n}}2}$ satisfy
 $[F_{\tilde{n},n_1}(s_1)] \cap
  [F_{\tilde{n},n_2}(s_2)] = \emptyset$,
 then all elements of $[F_{\tilde{n},n_1}(s_1)]$
 differ from all elements of
 $[F_{\tilde{n},n_2}(s_2)]$ before level $l$.

Let $g : \baire \to \baire$ be defined by
 $$g(x)(n) := \max \{
 \mbox{Exit}(T_n)(x),
 e(n),
 n \}.$$
Certainly $g$ is Borel,
 with a code in $M$ (because $\mc{W} \in M$).
Just like in the previous proposition,
 since $x_n \in [T_n]$,
 for all $x \in \baire$ and $n \in \omega$ we have
 $$x \not\in [T_n] \Rightarrow
 f_\mc{X}(x)(n) \le g(x)(n).$$
Suppose, towards a contradiction, that there is
 some $x \in \baire$ satisfying
 $f_\mc{X}(x) \not\le^* g(x)$.
Fix such an $x$.
Let $A$ be the infinite set
 $$A := \{ n \in \omega :
 f_\mc{X}(x)(n) > g(x)(n) \}.$$
It must be that $x \in [T_n]$ for each $n \in A$.
Since $A$ is infinite, we may fix
 $n_1, n_2 \in A$ satisfying the following:
\begin{itemize}
\item[i)] $n_1 < n_2$;
\item[ii)] $f_\mc{X}(x)(n_1) \le n_2$.
\end{itemize}
Let $\tilde{n}$ satisfy
 $\eta(\tilde{n}) = \{n_1, n_2\}$.
Since $x \in [T_{n_1}]$, fix some
 $s_1 \in {^{\tilde{n}}2}$ satisfying
 $$x \in [F_{\tilde{n},n_1}(s_1)] =: C_1.$$
Also, since $x_{n_2} \in [T_{n_2}]$, fix some
 $s_2 \in {^{\tilde{n}}2}$ satisfying
 $$x_{n_2} \in [F_{\tilde{n},n_2}(s_2)] =: C_2.$$

Because $f_\mc{X}(x)(n_2) > g(x)(n_2)$,
 we have $f_\mc{X}(x)(n_2) > e(n_2)$,
 so $$\mbox{Exit}([[x_{n_2}]])(x) > e(n_2).$$
This, combining with the definition of $e(n_2)$
 and the fact that $x \in C_1$ and
 $x_{n_2} \in C_2$ tells us that
 $C_1 \cap C_2 \not= \emptyset$
 (because otherwise $x \in C_1$ and
 $x_{n_2} \in C_2$ would differ before level
 $e(n_2)$, which by definition of
 $e(n_2)$ would mean that
 $\mbox{Exit}([[x_{n_2}]])(x) \le e(n_2)$).
Thus, by part 3) of the definition
 of a separation device,
 one of the following holds:
\begin{itemize}
\item[a)]
 $x_{n_1} = x_{n_2}$;
\item[b)]
 $C_1 = C_2 = \{x\}$.
\end{itemize}
Now, b) cannot be the case because
 $C_2 = \{x\}$ implies $x_{n_2} = x$,
 which implies $f_\mc{X}(x)(n_2) = 0$,
 which contradicts the fact that
 $f_\mc{X}(x)(n_2) > g(x)(n_2)$.
On the other hand,
 a) cannot be the case because
 $x_{n_1} = x_{n_2}$ implies
 $f_\mc{X}(x)(n_1) =
  f_\mc{X}(x)(n_2)$,
 which by ii) implies
 $$f_\mc{X}(x)(n_2) =
 f_\mc{X}(x)(n_1) \le
 n_2 \le
 g(x)(n_2) <
 f_\mc{X}(x)(n_2),$$
 which is impossible.
\end{proof}

\section{Sacks Forcing}

In this section, we will show that the
 shrink wrapping property holds between
 the ground model and any Sacks forcing extension.
 To learn more about Sacks forcing, the reader may
 consult \cite{Baumgartner}, \cite{Miller}, \cite{Hrusak} 
 and \cite{Geshke}.

\begin{definition}
A tree $p \subseteq {^{<\omega}2}$ is \defemp{perfect}
 iff it is nonempty and for each $t \in p$,
 there are incompatible $t_1, t_2 \in p$ extending $t$.
Sacks forcing $\mbb{S}$ is the poset of all
 perfect trees $p \subseteq \cantornodes$,
 where $p_1 \le p_2$ iff $p_1 \subseteq p_2$.
\end{definition}

Given $p_1, p_2 \in \mbb{S}$,
 $p_1 \perp p_2$ means that $p_1$ and $p_2$
 are incompatible.

\begin{definition}
Let $p \subseteq \bairenodes$ be a tree.
A node $t \in p$ is called a \defemp{branching node}
 iff $t$ has at least two (immediate) successors in $p$.
If there is a branching node,
 then $\mbox{Stem}(p)$ is defined as the minimal such node.
If $p$ is a leafless tree but has no branching node
 (so $p$ is \textit{not} perfect),
 then $\mbox{Stem}(p)$ is defined to be the unique
  path through $p$.
This will be convenient later.
A node $t \in p$ is said to be an
 \defemp{$n$-th branching node} iff it is a branching
 node and there are exactly $n$ branching nodes
 that are proper initial segments of it.
In particular, $\mbox{Stem}(p)$ is the unique
 $0$-th branching node of $p$.
Given Sacks conditions $p, q$, we write
 $q \le_n p$ iff $q \le p$ and all
 of the $k$-th branching nodes, for $k \le n$,
 of $p$ are in $q$ and are branching nodes.
\end{definition}

\begin{lemma}[Fusion Lemma]
Let $\langle p_n : n \in \omega \rangle$
 be a sequence of Sacks conditions such that
 $$p_0 \ge_0 p_1 \ge_1 p_2 \ge_2 ....$$
Then
 $p_\omega := \bigcap_{ n \in \omega } p_n$
 is a Sacks condition below each $p_n$.
\end{lemma}
\begin{proof}
This is standard and can be found
 in introductory presentations
 of Sacks forcing.
See, for example,
 \cite{Jech}.
\end{proof}

The sequence $\langle p_n : n \in \omega \rangle$
 in the lemma above is known as a
 \textit{fusion sequence}.
The following will help in the construction
 of fusion sequences.

\begin{lemma}[Fusion Helper Lemma]
\label{fusionhelper}
Let $R: \cantornodes \to \mbb{S}$ be a function
 with the following properties:
\begin{itemize}
\item[1)] $(\forall s_1, s_2 \in \cantornodes)$
 $s_2 \sqsupseteq s_1$ implies
 $R(s_2) \le R(s_1)$;
\item[2)] $(\forall s \in \cantornodes)$
 $\mbox{Stem}(R(s ^\frown 0)) \perp
 \mbox{Stem}(R(s ^\frown 1))$.
\end{itemize}
For each $n \in \omega$, let $p_n$ be the Sacks condition
 $$p_n := \bigcup \{ R(s) : s \in {{^n}2} \}.$$
Then $$R(\emptyset) = p_0 \ge p_1 \ge_0
 p_2 \ge_1 p_3 \ge_2 ...$$
 is a fusion sequence.
\end{lemma}
\begin{proof}
Consider any $n \ge 1$.
Certainly $p_n \supseteq p_{n+1}$,
 because for each $s \in {{^n}2}$,
 $R(s) \supseteq
 R(s ^\frown 0) \cup
 R(s ^\frown 1)$.
To show that $p_n \ge_{n-1} p_{n+1}$,
 consider a $k$-th branching
 node $t$ of $p_n$
 for some $k \le n-1$.
One can check that there is some
 $s \in {{^k}2}$ such that
 $t$ is the largest common initial segment of
 $\mbox{Stem}(R(s ^\frown 0))$ and
 $\mbox{Stem}(R(s ^\frown 1))$.
Since $$
 \mbox{Stem}(R(s ^\frown 0)) \cup
 \mbox{Stem}(R(s ^\frown 1))
 \subseteq R(s ^\frown 0) \cup R(s ^\frown 1)
 \subseteq p_{n+1},$$
 we have that $t$ is a branching node of $p_{n+1}$.
Thus, we have shown that for each $k \le n-1$,
 each $k$-th branching node of $p_n$ is
 a branching node of $p_{n+1}$.
Hence, $p_n \ge_{n-1} p_{n+1}$.
\end{proof}

We present
 a forcing lemma
 that is a basic building
 block for separating $x_{n_1}$
 from $x_{n_2}$.
Combining this with
 a fusion argument
 gives us the result.

\begin{lemma}
\label{differentdrum}
Let $\mbb{P}$ be any forcing.
Let $p_0, p_1 \in \mbb{P}$ be conditions.
Let $\dot{\tau}_0, \dot{\tau}_1$ be names
 for elements of $\baire$.
Suppose that there is no
 $x \in \baire$ satisfying
 the following two statements:
\begin{itemize}
\item[1)] $p_0 \forces \dot{\tau}_0 = \check{x}$;
\item[2)] $p_1 \forces \dot{\tau}_1 = \check{x}$.
\end{itemize}
Then there exist
 $p_0' \le p_0$;
 $p_1' \le p_1$;
 and $t_0, t_1 \in \bairenodes$
 satisfying the following:
\begin{itemize}
\item[3)] $t_0 \perp t_1$,
\item[4)] $p_0' \forces \dot{\tau}_0 \sqsupseteq \check{t}_0$,
\item[5)] $p_1' \forces \dot{\tau}_1 \sqsupseteq \check{t}_1$.
\end{itemize}
\end{lemma}
\begin{proof}
There are two cases to consider.
The first is that there exists some $x \in \baire$
 such that
 1) is true.
When this happens,
 2) is false.
Hence, there exist $t_1 \in \bairenodes$ and
 $p_1' \le p_1$ such that
 5) is true and
 $x \not\sqsupseteq t_1$.
Letting $p_0' := p_0$ and
 $t_0$ be some initial segment of $x$
 incompatible with $t_1$,
 we see that 3) and 4) are true.

The second case is that there is no
 $x \in \baire$ satisfying 1).
When this happens,
 there exist conditions $p_0^a,p_0^b \le p_0$
 and incompatible nodes
 $s_a, s_b \in \bairenodes$ satisfying both
 $p_0^a \forces \dot{\tau}_0 \sqsupseteq \check{s}_a$ and
 $p_0^b \forces \dot{\tau}_0 \sqsupseteq \check{s}_b.$
Now, it cannot be that both
 $p_1 \forces \dot{\tau}_1 \sqsupseteq \check{s}_a$ and
 $p_1 \forces \dot{\tau}_1 \sqsupseteq \check{s}_b$.
Assume, without loss of generality, that
 $p_1 \not\forces \dot{\tau}_1 \sqsupseteq \check{s}_a$.
This implies that there exist
 $p_1' \le p_1$ and $t_1 \in \bairenodes$ such that
 $s_a \perp t_1$ and
 $p_1' \forces \dot{\tau}_1 \sqsupseteq \check{t}_1$.
Letting $p_0' := p_0^a$
 and $t_0 := s_a$,
 we are done.
\end{proof}

At this point,
 the reader may want to think about how
 to use this lemma
 to prove that if $V$ is a Sacks forcing extension
 of a transitive model $M$ of $\zf$
 and $\mc{X} = \langle x_n : n \in \omega \rangle$ is a sequence of reals that  satisfies
 $$(\forall n \in \omega)\, x_n \not\in M$$ and
 $$\{ ( n_1, n_2 ) : x_{n_1} = x_{n_2} \} \in M,$$
 then there is a sequence $\mc{T}$ of subtrees of
 $\bairenodes$
 satisfying the hypotheses of Proposition~\ref{treeavoid}.

The next lemma explains the appearance
 of $I$ in the definition of a
 shrink wrapper.
We are intending the name $\dot{\tau}$
 to be such that
 $\dot{\tau}(n)$ refers to the $x_n$
 in the sequence
 $\mc{X} = \langle x_n : n \in \omega \rangle$.

\begin{lemma}
Consider Sacks forcing $\mbb{S}$.
Let $p \in \mbb{S}$ be a condition
 and $\dot{\tau}$ a name satisfying
 $p \forces \dot{\tau} : \omega \to \baire$.
Then there exists a condition $p' \le p$ and
 there exists a function $I: \omega \to [\baire]^{<\omega}$ satisfying
 $$p' \forces (\forall n \in \omega)\,
  \dot{\tau}(n) \in \check{V} \rightarrow
  \dot{\tau}(n) \in \check{I}(n).$$
\end{lemma}
\begin{proof}
We may easily construct a function
 $R: {^{<\omega}2} \to \mbb{S}$ that satisfies
 the conditions of Lemma~\ref{fusionhelper}
 such that $R(\emptyset) \le p$
 and for each $s \in {^n 2}$,
 either $R(s) \forces \dot{\tau}(n) \not\in \check{V}$
 or $(\exists x \in \baire)\, R(s) \forces \dot{\tau}(n) = \check{x}.$
Define $I$ as follows:
 $$I(n) := \{ x \in \baire :
 (\exists s \in {{^n}2})\, R(s) \forces \dot{\tau}(n) = \check{x} \}.$$
Let $p' := \bigcap_n \bigcup \{ R(s) : s \in {{^n}2} \}$.
The condition $p'$ and the function $I$ are as desired.
\end{proof}

We are now ready for the main
 forcing argument of this section.
\begin{theorem}
\label{mainsacksforcingargument}
Consider Sacks forcing $\mbb{S}$.
Let $p \in \mbb{S}$ be a condition and
 $\dot{\tau}$ be a name satisfying
 $p \forces \dot{\tau} : \omega \to \baire$.
Then there exists a condition $q \le p$
 and there exists
 $\mc{W} = \langle \mc{F}, I \rangle$ satisfying
 $$q \forces \check{\mc{W}}
 \mbox{ is a shrink wrapper for }
 \langle \dot{\tau}(n) : n \in \omega \rangle.$$
\end{theorem}
\begin{proof}
First, let $p' \le p$
 and $I : \omega \to [\baire]^{<\omega}$
 be given by the lemma above.
That is, for each $n \in \omega$,
 $$p' \forces
 \dot{\tau}(\check{n}) \in \check{V} \rightarrow
 \dot{\tau}(\check{n}) \in \check{I}(\check{n}).$$
We will define a function
 $R : {^{<\omega}2} \to \mbb{S}$
 with $R(\emptyset) \le p'$
 satisfying conditions 1) and 2) of
 Lemma~\ref{fusionhelper}.
At the same time, we will construct
 a family of functions
 $$\mc{F} = \langle F_{\tilde{n},n} :
 \tilde{n} \in \omega, n \in \eta(\tilde{n}) \rangle.$$
Let $\mc{W} = \langle \mc{F}, I \rangle$.
Our $q$ will be
 $$q := \bigcap_{\tilde{n}}
 \bigcup_{s \in {^{\tilde{n}}2}} R(s).$$

The function $F_{\tilde{n},n}$
 will return leafless subtrees of $\bairenodes$.
Moreover, each tree
 $F_{\tilde{n},n}(s)$ will obey the function
 $f_{\Phi(s,n)}$.
We will have it so
 for all $n \in \omega$
 and all $\tilde{n}$ satisfying $n \in \eta(\tilde{n})$,
 $$(\forall s \in {^{\tilde{n}}2})\,
 R(s) \forces \dot{\tau}(\check{n}) \in
 [\check{F}_{\tilde{n},n}(\check{s})].$$
Thus, $q$ will force
 that $\mc{W}$ satisfies conditions 1) and 2)
 of the definition of a shrink wrapper.
To show that $q$ forces condition 3)
 of that definition, it suffices 
 to show that for all $\{n_1, n_2\} = \eta(\tilde{n})$
 and all $s_1, s_2 \in {^{\tilde{n}}2}$,
 one of the following holds,
 where $T_1 := F_{\tilde{n},n_1}(s_1)$ and
 $T_2 := F_{\tilde{n},n_2}(s_2)$:
\begin{itemize}
\item[3a$'$)] $T_1 = T_2$ and $(\forall s \in {^{\tilde{n}}2})$,
 $$R(s) \forces
  (\dot{\tau}(\check{n}_1) \in [\check{T}_1] \vee
   \dot{\tau}(\check{n}_2) \in [\check{T}_2]) \rightarrow
   \dot{\tau}(\check{n}_1) =
   \dot{\tau}(\check{n}_2) ;$$
\item[3b$'$)] $(\exists x \in I(n_1) \cap I(n_2))\,
 [T_1] = [T_2] = \{ x \}$;
\item[3c$'$)] $[T_1] \cap [T_2] = \emptyset$, and moreover
 $\mbox{Stem}(T_1) \perp \mbox{Stem}(T_2)$.
\end{itemize}

We will define the functions $F_{\tilde{n},n}$
 and the conditions $R(s)$ for $s \in {^{\tilde{n}}2}$
 by induction on $\tilde{n}$.
Beginning at $\tilde{n} = 0$,
 let $\{ n_1, n_2 \} = \eta(0)$.
We will define
 $F_{0,n_1}$,
 $F_{0,n_2}$, and
 $R(\emptyset) \le p'$.

If $p' \forces
 \dot{\tau}(\check{n}_1) =
 \dot{\tau}(\check{n}_2)$,
 then using the Sacks property
 let $R(\emptyset) \le p'$ and
 $F_{0,n_1}(\emptyset) =
  F_{0,n_2}(\emptyset) = T$,
 where $T \subseteq \bairenodes$ is
 a tree that obeys both
 $f_{\Phi(\emptyset,n_1)}$ and
 $f_{\Phi(\emptyset,n_2)}$,
 such that
 $R(\emptyset) \forces \dot{\tau}(\check{n}_1) \in [\check{T}]$.
This causes 3a$'$) to be satisfied.

If $p' \not\forces \dot{\tau}(\check{n}_1)
 = \dot{\tau}(\check{n}_2)$,
 then let $t_1, t_2 \in \bairenodes$ be incomparable nodes
 and let $p_1 \le p'$ satisfy
 $p_1 \forces \dot{\tau}(\check{n}_1)
 \sqsupseteq \check{t}_1$ and
 $p_1 \forces \dot{\tau}(\check{n}_2)
 \sqsupseteq \check{t}_2$.
Then, using the Sacks property, we may define
 $R(\emptyset) \le p_1$ and
 $F_{0,n_1}(\emptyset) = T_1$ and
 $F_{0,n_2}(\emptyset) = T_2$ where
 $T_1$ and $T_2$ are leafless trees
 that obey
 $f_{\Phi(\emptyset,n_1)}$ and
 $f_{\Phi(\emptyset,n_2)}$ respectively
 such that
 $\mbox{Stem}(T_1) \sqsupseteq t_1$,
 $\mbox{Stem}(T_2) \sqsupseteq t_2$,
 $R(\emptyset) \forces
 \dot{\tau}(\check{n}_1) \in [\check{T}_1]$, and
 $R(\emptyset) \forces
 \dot{\tau}(\check{n}_2) \in [\check{T}_2]$.
This causes
 3c$'$) to be satisfied.

We will now handle the successor step of
 the induction.
Let $\{ n_1, n_2 \} = \eta(\tilde{n})$
 for some $\tilde{n} > 0$.
We will define
 $R(s)$ for each $s \in {^{\tilde{n}}2}$,
 and we will define both
 $F_{\tilde{n},n_1}$ and
 $F_{\tilde{n},n_2}$ assuming
 $R(s')$ has been defined
 for each $s' \in {^{< \tilde{n}}2}$.
To keep the construction readable,
 we will start with initial values for the
 $R(s)$'s and the $F_{\tilde{n},n}$'s,
 and we will modify them as the construction
 progresses until we arrive at their final values.
That is, we will say
 ``replace $R(s)$ with a stronger condition...'' and
 ``shrink the tree $F_{\tilde{n},n}(s)$...''.
When we make these replacements,
 it is understood that still
 $R(s) \forces \dot{\tau}(\check{n}) \in
 [\check{F}_{\tilde{n},n}(\check{s})]$.
The construction consists of 5 steps.

\underline{Step 1}:
 Let $R(s)$ for $s \in {^{\tilde{n}} 2}$
 and $F_{n,\tilde{n}}(s)$ for $s \in {^{\tilde{n}} 2}$
 and $n \in \{ n_1, n_2 \}$
 satisfy the following:
\begin{itemize}
\item[1)] For each $s' \in {^{(\tilde{n}-1)}2}$,
 the conditions $R({s'} ^\frown 0)$ and $R({s'} ^\frown 1)$
 both extend $R({s'})$ and they satisfy
 $\mbox{Stem}(R({s'} ^\frown 0)) \perp
 \mbox{Stem}(R({s'} ^\frown 1))$.
\item[2)] For $s \in {^{\tilde{n}} 2}$
 and $n \in \{ n_1, n_2 \}$,
 $F_{n,\tilde{n}}(s)$ is a leafless subtree
 of $\bairenodes$ that obeys $f_{\Phi(s,n)}$ and
 $R(s) \forces \dot{\tau}(\check{n}) \in [\check{F}_{\tilde{n},n}(\check{s})]$
\end{itemize}
Condition 2) can be arranged by the Sacks property.

\underline{Step 2}:
For each $s \in {^{\tilde{n}}2}$
 and $n \in \{n_1, n_2\}$,
 strengthen $R(s)$ so that either
 $R(s) \forces \dot{\tau}(\check{n}) \not\in \check{V}$ or
 $(\exists x \in I(n))\,
 R(s) \forces \dot{\tau}(\check{n}) = \check{x}$.
If the latter case holds,
 shrink $F_{\tilde{n},n}(s)$ so that
 it has only one path.

\underline{Step 3}:
For this step, fix $n \in \{ n_1, n_2 \}$.
For each pair
 of distinct $s_1, s_2 \in {^{\tilde{n}}2}$,
 strengthen each $R(s_1)$ and $R(s_2)$ and
 shrink each $F_{\tilde{n},n}(s_1)$ and
 $F_{\tilde{n},n}(s_2)$ so that
 one of the following holds:
\begin{itemize}
\item[i)] $(\exists x \in I(n))\,
 [F_{\tilde{n},n}(s_1)] =
 [F_{\tilde{n},n}(s_2)] =
 \{ x \}$;
\item[ii)]
 $\mbox{Stem}(F_{\tilde{n},n}(s_1)) \perp
  \mbox{Stem}(F_{\tilde{n},n}(s_2)).$
\end{itemize}
That is, if i) cannot be satisfied,
 then we may use Lemma~\ref{differentdrum}
 to satisfy ii).

\underline{Step 4}:
For each pair of distinct
 $s_1, s_2 \in {^{\tilde{n}}2}$ such that
 either $R(s_1) \forces
 \dot{\tau}(\check{n}_1) \not\in \check{V}$
 or $R(s_2) \forces
 \dot{\tau}(\check{n}_2) \not\in \check{V}$,
 use Lemma~\ref{differentdrum}
 to strengthen $R(s_1)$ and $R(s_2)$ and shrink
 $F_{\tilde{n},n_1}(s_1)$ and
 $F_{\tilde{n},n_2}(s_1)$ so that
 $$\mbox{Stem}(F_{\tilde{n},n_1}(s_1)) \perp
   \mbox{Stem}(F_{\tilde{n},n_2}(s_2)).$$

\underline{Step 5}:
For each $s \in {^{\tilde{n}}2}$,
 do the following:
If $R(s) \forces \dot{\tau}(\check{n}_1)
 = \dot{\tau}(\check{n}_2)$, then
 replace both
 $F_{\tilde{n},n_1}(s)$ and
 $F_{\tilde{n},n_2}(s)$ with
 $F_{\tilde{n},n_1}(s) \cap
  F_{\tilde{n},n_2}(s)$.
Otherwise,
 strengthen $R(s)$
 and shrink
 $F_{\tilde{n},n_1}(s)$ and
 $F_{\tilde{n},n_2}(s)$ so that
 $$\mbox{Stem}(F_{\tilde{n},n_1}(s)) \perp
   \mbox{Stem}(F_{\tilde{n},n_2}(s)).$$

This completes the construction
 of $\{ R(s) : s \in {^{\tilde{n}} 2} \}$,
 $F_{\tilde{n},n_1}$, and
 $F_{\tilde{n},n_2}$.
We will now prove that it works.
Fix $\tilde{n} \in \omega$ and
 $s_1, s_2 \in {^{\tilde{n}}2}$.
Let $T_1 := F_{\tilde{n},n_1}(s_1)$
 and $T_2 := F_{\tilde{n},n_2}(s_2)$.
We must show that one of
 3a$'$),
 3b$'$), or
 3c$'$) holds.
The cleanest way to do this is
 to break into cases depending on
 whether or not $s_1 = s_2$.

\underline{Case $s_1 \not= s_2$}:
If either
 $R(s_1) \forces \dot{\tau}(\check{n}_1) \not\in \check{V}$ or
 $R(s_2) \forces \dot{\tau}(\check{n}_2) \not\in \check{V}$,
 then by Step 4, we see that 3c$'$) holds.
Otherwise,
 by Step 2,
 $(\exists x \in I(n_1))\, [T_1] = \{x\}$ and
 $(\exists x \in I(n_1))\, [T_2] = \{x\}$.
Hence, either
 3b$'$) or 3c$'$) holds.

\underline{Case $s_1 = s_2$}:
If $R(s_1) \not\forces
 \dot{\tau}(\check{n}_1) =
 \dot{\tau}(\check{n}_2)$,
 then by Step 5, we see that 3c$'$) holds.
Otherwise, we are in the case that
 $$R(s_1) \forces
 \dot{\tau}(\check{n}_1) =
 \dot{\tau}(\check{n}_2).$$
By Step 5, $T_1 = T_2$.
Now, if
 $R(s_1) \forces
 \dot{\tau}(\check{n}_1) \in \check{V}$,
 then of course also
 $R(s_1) \forces
 \dot{\tau}(\check{n}_2) \in \check{V}$,
 and by Step 2) we see that
 3b$'$) holds.
Otherwise, $R(s_1) \forces
 \dot{\tau}(\check{n}_1) \not\in \check{V}$.
Hence, $[T_1]$ is not a singleton.
We will show that 3a$'$) holds.
Consider any
 $s \in {^{\tilde{n}}2}$.
We must show
 $$R(s) \forces
 (\dot{\tau}(\check{n}_1) \in [\check{T}_1]  \vee
  \dot{\tau}(\check{n}_2) \in [\check{T}_1]) \rightarrow
 \dot{\tau}(\check{n}_1) =
 \dot{\tau}(\check{n}_2).$$
If $s = s_1$, we are done.
Now suppose $s \not= s_1$.
It suffices to show
 $$R(s) \forces
 \neg( \dot{\tau}(\check{n}_1) \in [\check{T}_1] \vee
 \dot{\tau}(\check{n}_2) \in [\check{T}_1] ).$$
That is, it suffices to show
 $R(s) \forces
 \dot{\tau}(\check{n}_1) \not\in [\check{T}_1]$ and
 $R(s) \forces
 \dot{\tau}(\check{n}_2) \not\in [\check{T}_1].$
Since $s \not= s_1$ and
 $[T_1]$ is not a singleton,
 by Step 3,
 $\mbox{Stem}(F_{\tilde{n},n}(s)) \perp \mbox{Stem}(T_1)$.
Recall that
 $$R(s) \forces \dot{\tau}(\check{n}_1)
 \in [\check{F}_{\tilde{n},n}(\check{s})].$$
Hence,
 since $[\check{F}_{\tilde{n},n}(\check{s})]
 \cap [T_1] = \emptyset$,
 $R(s) \forces \dot{\tau}(\check{n}_1) \not\in
 [\check{T}_1]$.
By a similar argument,
 $R(s) \forces
 \dot{\tau}(\check{n}_2) \not\in [\check{T}_1]$.
This completes the proof.
\end{proof}

\section{Silver Forcing}

In this section,
 we will show that the shrink wrapping property
 does not hold between the ground model
 and any Silver forcing extension. To learn more
 about Silver forcing, the reader may consult \cite{Combinatorial}.
\begin{definition}
A tree $T \subseteq {^{<\omega}2}$
 is a \defemp{Silver} tree iff
 it is leafless and the following
 are satisfied.
There is an infinite set of levels
 $L \subseteq \omega$ such that for each
 $t \in T$,
 if $\dom(t) \in L$, then both
 $t ^\frown 0$ and $t ^\frown 1$ are in $T$,
 and if $\dom(t) \not\in L$, then exactly one of
 $t ^\frown 0$ or $t ^\frown 1$ is in $T$.
Also, if $x_1, x_2 \in [T]$ are two paths through $T$
 and $l \not\in L$, then
 $x_1(l) = x_2(l)$.
The poset of all Silver trees ordered by
 inclusion is called Silver forcing
 $\mbb{V}$.
\end{definition}

\begin{fact}
Suppose $G$ is $\mbb{V}$-generic over $V$.
Let $g = \bigcap G$. Then
 $$\{ T \in \mbb{V} : g \in [T] \} = G.$$
For this reason, we will sometimes say that
 $g$ is $\mbb{V}$-generic over $V$.
\end{fact}

\begin{definition}
Let $p \subseteq {^{<\omega}2}$ be a tree.
Let
 $t, s \in p$ be such that $\dom(t) = \dom(s)$.
When we say ``replace p below t with
 p below s'', we mean replace $p$ with
 $$\{ u \in p : u \not\sqsupseteq t \} \cup
 \{ t ^\frown w : s ^\frown w \in p \}.$$
That is, the subtree of $p$ below $s$
 is replacing the subtree of $p$ below $t$.
\end{definition}

In the following we will talk about elementary submodels
 of $V$, but we might as well be talking about elementary
 submodels of $V_\Theta \subseteq V$ for some large enough
 ordinal $\Theta$.

Recall that given a tree $T$ and a node $t$,
 the tree $T | t$ is the set of all
 nodes $s \in T$
 that are comparable to $t$.

\begin{lemma}
\label{allbranches}
Let $M$ be a countable elementary submodel of $V$
 and let $p \in \mbb{V}$ be in $M$.
Then there is some $p' \le p$
 (not in $M$) such that each branch
 through $p'$ is $\mbb{V}$-generic over $M$.
\end{lemma}
\begin{proof}
Let $\langle U_n : n \in \omega \rangle$
 be an enumeration of the dense subsets of
 $\mbb{V}$ that are in $M$.
We will define a decreasing sequence of conditions
 $p = p_{-1} \ge p_0 \ge p_1 \ge ...$
 in $M$.
Now fix $n \ge 0$ and
 suppose we have defined this sequence for
 $p_{-1} \ge ... \ge p_{n-1}$.
We will define $p_n$.
Let $\langle t_n^i : i < 2^n \rangle$
 be the nodes on the $n$-th splitting level
 of $p_{n-1}$.
First shrink $p_{n-1} | t^0_n$
 to be within $U_n$,
 calling the resulting condition
 $p_{n-1}^0$.
This shrinking is possible because
 $p_{n-1}$ is in $M$.
Then for each $i \not= 0$,
 replace $p_{n-1}^0$ below $t^i_n$ with
 $p_{n-1}^0$ below $t^0_n$.
Call the resulting condition
 $\tilde{p}_{n-1}^0$.
Then shrink
 $\tilde{p}_{n-1}^0 | t^1_n$ to be within $U_n$,
 calling the resulting condition
 $p_{n-1}^1$.
Then for each $i \not= 1$,
 replace $p_{n-1}^1$ below $t^i_n$ with
 $p^1_{n-1}$ below $t^1_n$.
Call the resulting condition
 $\tilde{p}_{n-1}^1$.
Continue this for all $i < 2^n$.
After all this shrinking,
 let $p_n := \tilde{p}_{n-1}^{2^n-1}$.
Now unfix $n$.
Note that $p_n \in \mbb{V}$.

We have now constructed the sequence
 $p = p_{-1} \ge p_0 \ge p_1 \ge ...$
 with the property that for each $n \in \omega$,
 each branch through $p_n$ is a path
 through some element of $U_n$.
Let $p' = \bigcap_{n \in \omega} p_n$.
Then each branch through $p'$
 is a branch through an element of each $U_n$.
Hence, each branch through $p'$ is
 $\mbb{V}$-generic over $M$.
\end{proof}

\begin{theorem}
Consider Silver forcing $\mbb{V}$.
There is some $\mc{\dot{X}}$ such that
 there is no $p$ and $\mc{W}$
 such that
 $p \forces \check{\mc{W}}$
 is a shrink wrapper for $\dot{\mc{X}}$.
\end{theorem}
\begin{proof}
Given a function $r : \omega \to 2$
 and $n \in \omega$,
 let
 $\mbox{Flatten}(r,n) : \omega \to 2$ be the function
 $$\mbox{Flatten}(r,n)(i) :=
 \begin{cases}
 0 & \mbox{if } i \le n, \\
 r(i) & \mbox{otherwise.}
 \end{cases}$$
Let $\dot{r}$ be the canonical name
 for the generic real.
We have $1 \forces \dot{r} : \omega \to 2$.
Let $\vec{0} \in {^\omega 2}$ be the
 constant zero function.
Let $\langle \dot{x}_n
 \in {^\omega 2} : n \in \omega \rangle$
 be a sequence of names such that for
 each $n \in \omega$,
 $$1 \forces \dot{x}_{2n} =
 \begin{cases}
 \mbox{Flatten}(\dot{r},n)
  & \mbox{if } \dot{r}(n) = 0, \\
 \vec{0} & \mbox{if } \dot{r}(n) = 1,
 \end{cases}$$
 and 
 $$1 \forces \dot{x}_{2n+1} =
 \begin{cases}
 \vec{0} & \mbox{if } \dot{r}(n) = 0, \\
 \mbox{Flatten}(\dot{r},n) &
 \mbox{if } \dot{r}(n) = 1.
 \end{cases}$$
That is, one of
 $\dot{x}_{2n}$ and $\dot{x}_{2n+1}$ will
 be a final segment of the generic real
 with initial zeros, and the other will
 the constant zero function.
Define $\dot{\mc{X}}$ such that
 $$1 \forces \dot{\mc{X}} =
 \langle \dot{x}_n : n \in \omega \rangle.$$
Suppose there is some condition $p$
 and some $\mc{W} = \langle \mc{F}, I \rangle$
 such that
 $$p \forces \check{\mc{W}}
 \mbox{ is a shrink wrapper for } \dot{\mc{X}}.$$
We will find a contradiction.

Let $M$ be a countable elementary substructure
 of $V$ such that $p, \mc{W},
 \dot{\mc{X}} \in M$.
By Lemma~\ref{allbranches}, let $p' \le p$
 be such that all branches through $p'$ are
 $\mbb{V}$-generic over $M$.
Let $n := |\mbox{Stem}(p')|$.
Let $\tilde{n} \in \omega$ be such that
 $\{ 2n, 2n+1 \} = \eta(\tilde{n})$.
That is, $\{ 2n, 2n+1 \}$ is the
 $\tilde{n}$-th pair.

Let $r_0$ be the leftmost
 branch through $p' | (\mbox{Stem}(p') ^\frown 0)$
 and let $r_1$ be the
 leftmost branch through
 $p' | (\mbox{Stem}(p') ^\frown 1)$.
Hence, $r_0(l) = r_1(l)$ for all $l \not= n$.
Let $u : \omega \to 2$ be such that
 $$u =
  \mbox{Flatten}(r_0, n) =
  \mbox{Flatten}(r_1, n).$$
Note that $u \not\in M$.

Now,
 $$M \models p \forces
  \check{W} \mbox{ is a shrink wrapper for }
  \dot{\mc{X}}.$$
Given a name $\dot{\tau}$ and a generic filter
 $G$, let $\dot{\tau}_G$ refer to the valuation
 of $\dot{\tau}$ with respect to $G$.
Since $r_0$ and $r_1$ are both paths through
 $p'$, they are generic over $M$.
Thus, we have
 $$M[r_0] \models \mc{W}
 \mbox{ is a shrink wrapper for }
 \dot{\mc{X}}_{r_0}.$$
By part 2) of Definition~\ref{srdef}, we have
 $$M[r_0] \models (\exists s_1 \in {^{\tilde{n}} 2})
  (\dot{x}_{2n})_{r_0} \in [F_{\tilde{n},2n}(s_1)].$$
Fix this $s_1$.
Let $T_1 := F_{\tilde{n},2n}(s_1)$.
We have
 $$(\dot{x}_{2n})_{r_0}^{M[r_0]} \in [T_1].$$
Similarly, we have
 $$M[r_1] \models
 (\exists s_2 \in {^{\tilde{n}} 2})
 (\dot{x}_{2n+1})_{r_1} \in
 [F_{\tilde{n}, 2n+1}(s_2)].$$
Fix this $s_2$.
Let $T_2 = F_{\tilde{n}, 2n+1}(s_2)$.
We have
 $$(\dot{x}_{2n+1})_{r_1}^{M[r_1]} \in [T_2].$$
Here is the crucial part:
 by the definition of $r_0, r_1$, and $\dot{\mc{X}}$,
 since $r_0(n) = 0$ and $r_1(n) = 1$,
 we have
 $$(\dot{x}_{2n})_{r_0}^{M[r_0]} =
 \mbox{Flatten}(r_0, n) =
 u =
 \mbox{Flatten}(r_1, n) =
 (\dot{x}_{2n+1})_{r_1}^{M[r_1]}.$$
Thus, we have
 $$[T_1] \cap [T_2] \not= \emptyset.$$
Since $\mc{W} \in M[r_0]$, by absoluteness we have
 $$M[r_0] \models [T_1] \cap [T_2] \not= \emptyset.$$
Working in $M[r_0]$,
 by part 3) of Definition~\ref{srdef}
 applied to $C_1 := [T_1]$ and $C_2 := [T_2]$,
 it must be that either 3a) or 3b) holds.
Note that
 $$(\dot{x}_{2n})_{r_0}^{M[r_0]} = u \not\in M,$$
 which implies that 3b) cannot hold.
Since $(\dot{x}_{2n})_{r_0}^{M[r_0]} \in [T_1]$,
 using 3a) we have that
 $$M[r_0] \models
 (\dot{x}_{2n})_{r_0} =
 (\dot{x}_{2n+1})_{r_0}.$$
Thus,
 $$(\dot{x}_{2n})_{r_0}^{M[r_0]} =
 (\dot{x}_{2n+1})_{r_0}^{M[r_0]}.$$
We already know that the left hand side
 of this equation is $u$.
On the other hand, by definition,
 the right hand side must be $\vec{0}$.
This is a contradiction.
\end{proof}

\end{document}